\documentclass[12pt]{amsart}
\usepackage{amsmath, amsthm, amscd, amsfonts, amssymb, graphicx, color}
\usepackage{enumerate}
\usepackage[bookmarksnumbered, colorlinks, plainpages]{hyperref}
\usepackage{tikz}
\usepackage{comment}

\newtheorem{theorem}{Theorem}[section]
\newtheorem{lemma}[theorem]{Lemma}
\newtheorem{proposition}[theorem]{Proposition}

\newtheorem{corollary}[theorem]{Corollary}
\newtheorem{definition}[theorem]{Definition}

\newtheorem{remark}[theorem]{Remark}

\excludecomment{mysection}

\begin{document}
	
\title[Projections and orthogonality]{Projections in an order unit space and orthogonality}
\author{Anil Kumar Karn}
	
\address{School of Mathematical Sciences, National Institute of Science Education and Research Bhubaneswar, An OCC of Homi Bhabha National Institute, P.O. - Jatni, District - Khurda, Odisha - 752050, India.}

\email{\textcolor[rgb]{0.00,0.00,0.84}{anilkarn@niser.ac.in}}

	
	
\begin{abstract}
	We introduce the notion of order projections using the order unit property of a positive element in an order unit space and characterize them in terms of (geometric) orthogonality. We describe order projections of the order unit space obtained by adjoining an order unit to a normed linear space.
\end{abstract}

\maketitle 

\section{Introduction} 

Let $A$ be a unital C$^*$-algebra. Then $p \in A^+$ is said to be a \emph{projection}, if $p^2 = p$. Projections are one of the basic ingredients of operator algebra and play an important role in the classification theory. %

There have been many attempts to study projections in general set-ups, for example, in order unit spaces. Most of them have taken an algebraic discourse; namely embedding the space in its C$^*$-algebra envelop and those elements, which emerge as projections in the enveloping C$^*$-algebra, are termed as projections. 

In \cite{K4}, projections were described from a different point of view. 
\begin{definition}
	Let $(V, e)$ be an order unit space. Then $u \in V^+$ is said to have the \emph{order unit property} (OUP, in short), if for all $v \in V$, we have $\Vert v \Vert u \pm v \in V^+$ whenever $\lambda u \pm v \in V^+$. 
\end{definition}

It was shown in \cite[(Corollary 3.2 and Theorem 3.3.)]{K4} that in a (unital) C$^*$-algebra, the projections are the only positive elements having the (matrix) order unit property. Replicating the proof of \cite[Theorem 3.3]{K4}, we can extend it to a unital $JB$-algebra. Precisely stating, a positive element in a unital $JB$-algebra is a projection if and only if it has the order unit property.

In other words, projections in a C$^*$-algebra are more than just being algebraic in nature. More precisely, projections have several characterizing properties. Below we list some of them. Let $A$ be a unital C$^*$-algebra and let $p \in A^+$ with $\Vert p \Vert \le 1$. Then the following facts are equivalent: 
\begin{enumerate}
	\item $p$ is a projection.
	\item $p (1 - p) = 0$. (That is, $p$ is algebraic orthogonal to $1 - p$.) 
	\item $[0, p] \cap [0, 1 - p] = \lbrace 0 \rbrace$.
	\item $p$ is an extreme point of the convex set $[0, 1]$. 
	\item $p$ and $1 - p$ have the order unit property (definition given later) in $A$.
\end{enumerate}
 In \cite{K18}, it was proved that algebraic orthogonality in a C$^*$-algebra can be characterized via order theoretic orthogonality. 
\begin{definition}\cite{K14} 
	Let $X$ be a non-zero real normed linear space and let $x, y \in X$. We say that $x$ is 	$\infty$-orthogonal to $y$, (we write, $x \perp_{\infty} y$) if $\Vert x + k y \Vert = \max \lbrace \Vert x \Vert, \Vert k y \Vert \rbrace$ for all $k \in \mathbb{R}$.
\end{definition} 
	Let $A$ be a (unital) C$^*$-algebra and assume that $a, b \in A^+$. It follows from \cite{K16, K18} that $a b = 0$ if and only if $[0, a] \perp_{\infty} [0, b]$. Thus $p \in A^+$ with $\Vert p \Vert \le 1$, we conclude that $p$ is a projection if and only if $[0, p] \perp_{\infty} [0, 1 - p]$. 
	
	In this paper, we prove a stronger property of algebraic orthogonality: $a b = 0$ if and only if $[- a, a] \perp_{\infty} [- b, b]$. Interestingly, this form of orthogonality leads us to a characterization of the order unit property of a positive element in a unital C$^*$-algebra.
\begin{theorem}\label{oup}
	Let $(V, e)$ be an order unit space and assume that $u \in V^+ \setminus \lbrace 0, e \rbrace$. Then $u$ has the OUP in $V$ if and only if $\Vert u \Vert = 1 = \Vert e - u \Vert$ and $(e - u) \perp_{\infty} [- u, u]$ (that is, $(e - u) \perp_{\infty} v$ whenever $v \in V$ with $v \in [- u, u]$). 
\end{theorem} 
Let us recall from \cite{K23} that the set 
$$S_V := \lbrace u \in V^+: \Vert u \Vert = 1 = \Vert e - u \Vert \rbrace$$ 
together with $e$ determines $V$ as an order unit space. Note that $u \in S_V$, if it has the OUP in $V$ with $u \notin \lbrace 0, e \rbrace$. 

(Also, it is not difficult to prove that the elements having the OUP in $V$ are extreme points of the convex set $V_1^+ := \lbrace u \in V^+: \Vert u \Vert \le 1 \rbrace$. In a unital C$^*$-algebra, even the converse is also true. However, we can not comment about it in a general order unit space.) 

We can further expand Theorem \ref{oup}. We shall say that $u \in V^+$ is an \emph{order projection}, if both $u$ and $e - u$ have the OUP.
\begin{corollary}\label{poud}
	Let $(V, e)$ be an order unit space and assume that $u \in S_V$. Then the following statements are equivalent: 
	\begin{enumerate}
		\item $u$ is an order projection in $V$; 
		\item $[- u, u] \perp_{\infty} [- (e - u), (e - u)]$; 
		\item $u \perp_{\infty} [- (e - u), (e - u)]$ and $(e - u) \perp_{\infty} [- u, u]$. 
	\end{enumerate}
\end{corollary} 
The condition obtained in Corollary \ref{poud} leads to a notion of a stronger type of orthogonality. However, we prove that the two conditions coincide in any C$^*$-algebra and more generally, in any absolute order unit space. 

\section{Orthogonality in normed linear spaces} 

In this section, we find a criterion of $\infty$-orthogonality. For this, we recall and use the following notion.
\begin{definition}\cite{K14} 
	Let $X$ be a non-zero real normed linear space and let $x, y \in X$. We say that $x$ is $1$-orthogonal to $y$, (we write, $x \perp_1 y$) if $\Vert x + k y \Vert = \Vert x \Vert + \Vert k y \Vert$ for all $k \in \mathbb{R}$. %Similarly, we say that $x$ is $\infty$-orthogonal to $y$, (we write, $x \perp_{\infty} y$) if $\Vert x + k y \Vert = \max \lbrace \Vert x \Vert, \Vert k y \Vert \rbrace$ for all $k \in \mathbb{R}$.
\end{definition} 
The following result is apparently a folklore. We include a proof for the sake of completeness. 
\begin{proposition}\label{orth1}
	Let $X$ be a non-zero real normed linear space and let $x, y \in X$. Then $x \perp_1 y$ if and only if $\Vert x \pm y \Vert = \Vert x \Vert + \Vert y \Vert$. 
\end{proposition}
\begin{proof}
	If $x \perp_1 y$, then $\Vert x \pm y \Vert = \Vert x \Vert + \Vert y \Vert$. Conversely, we assume that $\Vert x \pm y \Vert = \Vert x \Vert + \Vert y \Vert$. Then 
	\begin{eqnarray*}
		\Vert x + k y \Vert &=& \Vert (x + y) - (1 - k) y \Vert \\ 
		&\ge& \Vert x + y \Vert - \Vert (1 - k) y \Vert \\ 
		&=& \Vert x \Vert + \Vert y \Vert - (1 - k) \Vert y \vert \\ 
		&=& \Vert x \Vert + \Vert k y \Vert
	\end{eqnarray*} 
	and 
	\begin{eqnarray*}
		\Vert x - k y \Vert &=& \Vert (x - y) + (1 - k) y \Vert \\ 
		&\ge& \Vert x - y \Vert - \Vert (1 - k) y \Vert \\ 
		&=& \Vert x \Vert + \Vert y \Vert - (1 - k) \Vert y \vert \\ 
		&=& \Vert x \Vert + \Vert k y \Vert. 
	\end{eqnarray*} 
	Also $\Vert x \pm k y \Vert \le \Vert x \Vert + \Vert k y \Vert$ so that $\Vert x \pm k y \Vert = \Vert x \Vert + \Vert k y \Vert$ whenever $0 \le k \le 1$. Further, for $k > 1$, we have 
	$$\Vert x \pm k y \Vert = k \Vert k^{-1} x \pm y \Vert = k (\Vert k^{-1} x \Vert + \Vert y \Vert) = \Vert x \Vert + \Vert k y \Vert.$$ 
	Hence $x \perp_1 y$.
\end{proof} 
We prove a similar characterization for $\infty$-orthogonality. To this end, we apply the following result that describes a duality between the two types of orthogonality.
\begin{lemma}\label{pd}
	Let $X$ be a non-zero real normed linear space and let $x, y \in X \setminus \lbrace 0 \rbrace$.
	\begin{enumerate}
		\item If $x \perp_1 y$, then $(\Vert x \Vert^{-1} x + \Vert y \Vert^{-1} y) \perp_{\infty} (\Vert x \Vert^{-1} x - \Vert y \Vert^{-1} y)$. 
		\item If $x \perp_{\infty} y$, then $(\Vert x \Vert^{-1} x + \Vert y \Vert^{-1} y) \perp_1 (\Vert x \Vert^{-1} x - \Vert y \Vert^{-1} y)$. 
	\end{enumerate}
\end{lemma} 
\begin{proof}
	Since $x \perp_p y$ if and only if $\Vert x \Vert^{-1} x \perp_p \Vert y \Vert^{-1} y$ for $p = 1$ or $\infty$, without any loss of generality, we may assume that $\Vert x \Vert = 1 = \Vert y \Vert$. 
	
	Let $x \perp_1 y$. Then $\Vert x \pm y \Vert = 2$. Thus for any $k \in \mathbb{R}$, we have 
	\begin{eqnarray*}
		\Vert (x + y) + k (x - y) \Vert &=& \Vert (1 + k) x + (1 - k) y \Vert \\
		&=& \Vert (1 + k) x \Vert + \Vert (1 - k) y \Vert \\ 
		&=& \vert 1 + k \vert + \vert 1 - k \vert \\ 
		&=& 1 + \vert k \vert + \vert 1 - \vert k \vert \vert \\ 
		&=& 2 \left( \max \lbrace 1, \vert k \vert \rbrace \right) \\ 
		&=& \max \lbrace \Vert x + y \Vert, \Vert k (x - y) \Vert \rbrace.
	\end{eqnarray*} 
	Hence $(x + y) \perp_{\infty} (x - y)$. 
	
	Next, we assume that $x \perp_{\infty} y$. Then $\Vert x \pm y \Vert = 1$. Thus for any $k \in \mathbb{R}$, we have 
	\begin{eqnarray*}
		\Vert (x + y) + k (x - y) \Vert &=& \Vert (1 + k) x + (1 - k) y \Vert \\
		&=& \max \lbrace \Vert (1 + k) x \Vert, \Vert (1 - k) y \Vert \rbrace \\ 
		&=& \max \lbrace \vert 1 + k \vert, \vert 1 - k \vert \rbrace \\ 
		%&=& 1 + \vert k \vert + \vert 1 - \vert k \vert \vert \\ 
		&=& 1 + \vert k \vert \\ 
		&=& \Vert x + y \Vert + \Vert k (x - y) \Vert.
	\end{eqnarray*} 
	Hence $(x + y) \perp_1 (x - y)$. 
\end{proof} 
\begin{proposition}\label{oinfty}
	Let $X$ be a non-zero real normed linear space and let $x, y \in X \setminus \lbrace 0 \rbrace$. Then $x \perp_{\infty} y$ if and only if $\left\Vert \Vert x \Vert^{-1} x \pm \Vert y \Vert^{-1} y \right\Vert = 1$. 
\end{proposition} 
\begin{proof}
	For simplicity and without any loss of generality, we assume that $\Vert x \Vert = 1 = \Vert y \Vert$. If $x \perp_{\infty} y$, then $\Vert x \pm y \Vert = 1$. Conversely, we assume that $\Vert x \pm y \Vert = 1$. Put $x + y = u$ and $x - y = v$. Then $\Vert u \Vert = 1 = \Vert v \Vert$. Also $u + v = 2 x$ and $u - v = 2 y$ so that $\Vert u \pm v \Vert = 2$. Thus by Proposition \ref{orth1}, we get $u \perp_1 v$. Now, by Lemma \ref{pd}(1), we have $(u + v) \perp_{\infty} (u - v)$, or equivalently, $x \perp_{\infty} y$. 
\end{proof} 

\section{Expanding orthogonality} 

We take a cue from Corollary \ref{poud} and introduce the following notion to strengthen  $\infty$-orthogonality. 
\begin{definition}
	Let $(V, e)$ be an order unit space and assume that $u, v \in V^+$. We say that $u$ is \emph{extensively $\infty$-orthogonal} to $v$, (we write $u \perp_{\infty}^e v$), if $[- u, u] \perp_{\infty} [- v, v]$. 
\end{definition} 
In \cite{K16, K18}, the author had introduced and studied the notion of \emph{absolutely $\infty$-orthogonality}. Let $(V, e)$ be an order unit space and assume that $u, v \in V^+$. We say that $u$ is \emph{absolutely $\infty$-orthogonal} to $v$ ($u \perp_{\infty}^a v$) if $[0, u] \perp_{\infty} [0, v]$. Thus extensive $\infty$-orthogonality is apparently stronger than absolute $\infty$-orthogonality. However, we show that the two notions coincide in any absolute order unit space. 
\begin{theorem}\label{e=a}
	Let $V$ be an absolute order unit space and assume that $u, v \in V^+$ with $u \perp_{\infty}^a v$. Then $u \perp_{\infty}^e v$. 
\end{theorem} 
\begin{proof}
	Let $u_0 \in [- u, u]$ and $v_0 \in [- v, v]$. Put $u_1 = \frac 12 (u + u_0)$, $u_2 = \frac 12 (u - u_0)$, $v_1 = \frac 12 (v + v_0)$ and $v_2 = \frac 12 (v - v_0)$. Then $u_1, u_2 \in [0, u]$ and $v_1, v_2 \in [0, v]$. Since $u \perp_{\infty}^a v$, we get that $u_i \perp_{\infty}^a v_j$ for $i, j \in \lbrace 1, 2 \rbrace$. Thus $u_i \perp_{\infty}^a \vert v_1 - v_2 \vert$ for $i = 1, 2$ so that $\vert u_1 - u_2 \vert \perp_{\infty}^a \vert v_1 - v_2 \vert$. In other words, $\vert u_0 \vert \perp_{\infty}^a \vert v_0 \vert$ for $u_1 - u_2 = u_0$ and $v_1 - v_2 = v_0$. Since  $\vert u_0 \vert \perp_{\infty}^a \vert v_0 \vert$ implies that $\lambda \vert u_0 \vert \perp_{\infty}^a \mu \vert v_0 \vert$ for all $\lambda, \mu \ge 0$, we may assume that $\Vert u_0 \Vert = \Vert \vert u_0 \vert \Vert = 1 = \Vert v_0 \Vert = \Vert \vert v_0 \vert \Vert$. Then $\Vert \vert u_0 \vert + \vert v_0 \vert \Vert = 1$. Now as $\pm u_0 \le \vert u_0 \vert$ and $\pm v_0 \le \vert v_0 \vert$, we have $\pm (u_0 \pm v_0) \le \vert u_0 \vert + \vert v_0 \vert$. Thus 
	$$\Vert u_0 \pm v_0 \Vert \le \Vert \vert u_0 \vert + \vert v_0 \vert \Vert = 1.$$ 
	Set $u_0 + v_0 = x$ and $u_0 - v_0 = y$ so that $\Vert x \Vert \le 1$ and $\Vert y \Vert \le 1$. Also $x + y = 2 u_0$ so that $\Vert x + y \Vert = 2$. Thus 
	$$2 = \Vert x + y \Vert \le \Vert x \Vert + \Vert y \Vert \le 2$$ 
	so that $\Vert x \Vert = 1 = \Vert y \Vert$. In other words, $\Vert u_0 \pm v_0 \Vert = 1$. Thus by Proposition \ref{oinfty}, we have $u_0 \perp_{\infty} v_0$. Since $u_0 \in [- u, u]$ and $v_0 \in [- v, v]$ are arbitrary, we conclude that $u \perp_{\infty}^e v$. 
\end{proof} 
\begin{corollary}
	Let $A$ be a C$^*$-algebra and assume that $a, b \in A^+$. Then the following facts are equivalent: 
	\begin{enumerate}
		\item $a \perp_{\infty}^e b$; 
		\item $a \perp_{\infty}^a b$; and 
		\item $a b = 0$.
	\end{enumerate}
\end{corollary} 
\begin{proof}
	Evidently, $(1)$ implies $(2)$. It was proved in \cite{K16, K18} that $(2)$ is equivalent to $(3)$. If $A$ is unital, then $A_{sa}$ is an absolute order unit space. Thus by Theorem \ref{e=a}, $(1)$ is equivalent to $(2)$. So let $A$ be non-unital. We only need to prove that $(3)$ implies $(1)$. 
	
	Assume that $a b = 0$. Let $c \in [- a, a]$ and $d \in [- b, b]$. Put $c_1 = \frac 12 (a + c)$, $c_2 = \frac 12 (a - c)$, $d_1 = \frac 12 (b + d)$ and $d_2 = \frac 12 (b - d)$. Then $c_1, c_2 \in [0, a]$, $d_1, d_2 \in [0, b]$. Also then $c = c_1 - c_2$ and $d = d_1 - d_2$. As $a b = 0$, we have $a \perp_{\infty}^a b$ so that $c_i \perp_{\infty}^a d_j$ for $i, j \in \lbrace 1, 2 \rbrace$. Thus $c_i d_j = 0$ for $i, j \in \lbrace 1, 2 \rbrace$ whence $c d = 0$. Let $k \in \mathbb{R}$. Then 
	\begin{eqnarray*}
		\Vert c + k d \Vert^2 &=& \Vert (c + k d)^2 \Vert \\
		&=& \Vert c^2 + k^2 d^2 \Vert \\ 
		&=& \max \lbrace \Vert c^2 \Vert, \Vert k^2 d^2 \Vert \rbrace %\\ 
		%&=& \left( \max \lbrace \Vert c \Vert, \Vert k d \Vert \rbrace \right)^2
	\end{eqnarray*} 
	so that $\Vert c + k d \Vert = \max \lbrace \Vert c \Vert, \Vert k d \Vert \rbrace$ for all $k \in \mathbb{R}$. Thus $c \perp_{\infty} d$. Since $c \in [-a, a]$ and $d \in [- b, b]$ are arbitrary, we have $a \perp_{\infty}^e b$. 
\end{proof}

\section{Order unit property}

In this section, we prove the main result. First, we prove Theorem \ref{oup} and for this purpose, we use the following result.
\begin{lemma}\label{oupl}
	Let $(V, e)$ be an order unit space and assume that $u \in V^+$. Then $u$ has the OUP in $V$ if and only if $\Vert \Vert v \Vert (e - u) \pm v \Vert = \Vert v \Vert$ whenever $v \in V$ with $v \in [- u, u]$. 
\end{lemma}
\begin{proof}
	Assume that $u$ has the OUP in $V$. Let $v \in V$ with $v \ne 0$ be such that $v \in [- u, u]$. As $u$ has OUP in $V$, we get  $\Vert v \Vert^{-1} v \in [- u, u]$. Then 
	$$e - \left( (e - u) \pm \Vert v \Vert^{-1} v \right) = u \mp \Vert v \Vert^{-1} v \in V^+.$$ 
	Also 
	$$e + \left( (e - u) \pm \Vert v \Vert^{-1} v \right) \ge e \pm \Vert v \Vert^{-1} v \ge 0$$ 
	so that $e \pm \left( (e - u) \pm \Vert v \Vert^{-1} v \right) \in V^+$. Therefore, $\Vert (e - u) \pm \Vert v \Vert^{-1} v \Vert \le 1$. Now as 
	$$2 = 2 \Vert \Vert v \Vert^{-1} v \Vert \le \Vert (e - u) + \Vert v \Vert^{-1} v \Vert + \Vert (e - u) - \Vert v \Vert^{-1} v \Vert \le 2,$$ 
	We must have $\Vert (e - u) \pm \Vert v \Vert^{-1} v \Vert = 1$. Thus $\Vert \Vert v \Vert (e - u) \pm v \Vert = \Vert v \Vert$ whenever $v \in V$ with $v \in [- u, u]$.
	
	Conversely, we assume that $\Vert \Vert v \Vert (e - u) \pm v \Vert = \Vert v \Vert$ whenever $v \in V$ with $v \in [- u, u]$. Let $v \in V$ be such that $\lambda u \pm v \in V^+$ for some $\lambda \in \mathbb{R}$ with $\lambda > 0$. Then $\lambda^{-1} v \in [- u, u]$. Thus by assumption, $\Vert \Vert v \Vert (e - u) \pm v \Vert = \Vert v \Vert$. It follows that $\Vert v \Vert (e - u) \pm v \le \Vert v \Vert e$ whence $\pm v \le \Vert v \Vert u$. Hence $u$ has the OUP in $V$. 
\end{proof}
\begin{proof} (of Theorem \ref{oup}.) 
	First, we assume that $u$ has the OUP in $V$. 
	
	{\bf Step 1.} $\Vert u \Vert = 1$:  
	
	As $u \in [- u, u]$, by Lemma \ref{oupl}, we have $\Vert \Vert u \Vert (e - u) \pm u \Vert = \Vert u \Vert$ so that $\Vert u \Vert e \pm (\Vert u \Vert (e - u) \pm u) \in V^+$. Thus $(\Vert u \Vert + 1) u \le 2 \Vert u \Vert e$, and consequently, $(\Vert u \Vert + 1) \Vert u \Vert \le 2 \Vert u \Vert$. Since $u \ne 0$, we may deduce that $\Vert u \Vert \le 1$. Again, as $u$ has the OUP in $V$, we also have that $u \le \Vert u \Vert u$. Thus $\Vert u \Vert \ge 1$. Therefore, $\Vert u \Vert = 1$. 
	
	{\bf Step 2.} $\Vert e - u \Vert = 1$:  
	
	By Step 1, we have $\Vert u \Vert = 1$. Thus by \cite{K23}, there exists $\bar{u} \in S_V$ and $\alpha \in [0, 1]$ such that $u = (1 - \alpha) e + \alpha \bar{u}$. If possible, let $\alpha < 1$. Then $\pm e \le (1 - \alpha)^{-1} u$. As $u$ has the OUP in $V$, we get $\pm e \le \Vert e \Vert u = u$. In particular, $e \le u$. Since $\Vert u \Vert = 1$, we have $u \le e$. Thus $u = e$, contradicting the assumption. Hence $\alpha = 1$ and consequently, $u = \bar{u} \in S_V$. 
	
	Now invoking  Proposition \ref{oinfty} in the above proof, we may conclude the result. 
\end{proof} 
\begin{proof} (of Corollary \ref{poud}.) 
	That $(2)$ implies $(3)$ is evident and $(3)$ implies $(1)$ by Theorem \ref{oup}. We prove that $(1)$ implies $(2)$. 
	
	Assume that $u$ and $e - u$ have OUP in $V$ and let $v \in [- u, u]$ and $w \in [- (e - u), (e - u)]$. By the OUP of $u$ and $e - u$, we have $\Vert v \Vert^{-1} v \in [- u, u]$ and $\Vert w \Vert^{-1} w \in [- (e - u), (e - u)]$. So for simplicity of the proof, we assume that $\Vert v \Vert = 1 = \Vert w \Vert$. 
	
	Since $v \in [- u, u]$ and $w \in [- (e - u), (e - u)]$, we have $u \pm v \in V^+$ and $e - u \pm w \in V^+$. Thus $e \pm (v \pm w) \in V^+$ so that $\Vert v \pm w \Vert \le 1$. Set $v + w = x$ and $v - w = y$. Then $\Vert x \Vert \le 1$ and $\Vert y \Vert \le 1$. Also $x + y = 2 v$ so that $\Vert x + y \Vert = 2$. Thus $2 = \Vert x + y \Vert \le \Vert x \Vert + \Vert y \Vert \le 2$ whence $\Vert x \Vert = 1$ and $\Vert w \Vert = 1$. In other words, $\Vert v \pm w \Vert = 1$. Now by Proposition \ref{oinfty}, $v \perp_{\infty} w$. Hence $[- u, u] \perp_{\infty} [- (e - u), (e - u)]$. 
\end{proof} 
\begin{remark}
	Let $(V, e)$ be an order unit space and assume that $u \in S_V$ be an order projection in $V$. Put 
	$$V_u := \lbrace v \in V: \lambda u \pm v \in V^+ ~ \mbox{for some} ~ \lambda > 0 \rbrace.$$ 
	Then $(V_u, u)$ is an order unit space whose order unit norm coincides with that of $V$. Moreover, it is a closed normed order ideal of $V$ as well. Similarly, $(\langle e - u \rangle, e - u)$ is also an order unit space as well as a closed normed order ideal of $V$. Next, put 
	$$[u] := V_u + V_{e - u}.$$ 
	Then $[u]$ is a closed order unit subspace of $V$ (containing $e$). Further, if $v \in V_u$ and $w \in V_{e - u}$, then $v \perp_{\infty}w$ so that 
	$$\Vert v + w \Vert = \max \lbrace \Vert v \Vert, \Vert w \Vert \rbrace.$$ 
	Thus $[u]$ is isometrically isomorphic to $V_u \oplus_{\infty} V_{e - u}$. 
\end{remark} 
\begin{proposition}\label{sum}
	Let $(V, e)$ be an order unit space and assume that $u, v, w \in S_V$ with $u = v + w$ such that $u$ has the OUP in $V$. Then $v$ and $w$ also have the OUP in $V$ if and only if $v \perp_{\infty}^e w$. 
\end{proposition} 
\begin{proof}
	First, we assume that $v \perp_{\infty}^e w$. Let $v_1 \in V$ be such that $\lambda v \pm v_1 \in V^+$ for some $\lambda > 0$. Then $\lambda^{-1} v_1 \in [- v, v]$. Since $v \perp_{\infty}^e w$, we get that $w \perp_{\infty} \lambda^{-1} v_1$. Thus $\Vert w \pm \Vert v_1 \Vert^{-1} v_1 \Vert = 1$. 
	
	Next, put $k = \max \lbrace \lambda \Vert v_1 \Vert^{-1}, 1 \rbrace$. Then 
	\begin{eqnarray*}
		k u + (w \pm \Vert v_1 \Vert^{-1} v_1) &=& (k + 1) w + (k v \pm \Vert v_1 \Vert^{-1} v_1) \\ 
		&\ge& \Vert v_1 \Vert^{-1} (\lambda v \pm v_1) \ge 0
	\end{eqnarray*} 
	and 
	\begin{eqnarray*}
		k u - (w \pm \Vert v_1 \Vert^{-1} v_1) &=& (k - 1) w + (k v \mp \Vert v_1 \Vert^{-1} v_1) \\ 
		&\ge& \Vert v_1 \Vert^{-1} (\lambda v \mp v_1) \ge 0.
	\end{eqnarray*} 
	Since $u$ has OUP in $V$, we conclude that $u \pm (w \pm \Vert v_1 \Vert^{-1} v_1) \in V^+$ for $\Vert w \pm \Vert v_1 \Vert^{-1} v_1 \Vert = 1$. Thus $v \pm \Vert v_1 \Vert^{-1} v_1 \in V^+$. Since $v_1$ is arbitrary, we conclude that $v$ has OUP in $V$. Following in this way, we can now show that $w$ also has the OUP in $V$. 
	
	Conversely, we assume that $v$ and $w$ have OUP in $V$. Let $v_1 \in [- v, v]$ and $w_1 \in [- w, w]$. Since $v$ and $w$ have OUP in $V$, we get $\Vert v_1 \Vert^{-1} v_1 \in [- v, v]$ and $\Vert w_1 \Vert^{-1} w_1 \in [- w, w]$. So we may assume that $\Vert v_1 \Vert = 1 = \Vert w_1 \Vert$. Now $v_1 \pm w_1 \in [- u, u]$ so that $\Vert v_1 \pm w_1 \Vert \le 1$. Thus 
	\begin{eqnarray*}
		2 &=& \Vert 2 v_1 \Vert \\ 
		&=& \Vert (v_1 + w_1) + (v_1 - w_1) \Vert \\ 
		&\le& \Vert v_1 + w_1 \Vert + \Vert v_1 - w_1 \Vert \\ 
		&\le& 2
	\end{eqnarray*} 
	which leads to $\Vert v_1 + w_1 \Vert = 1 = \Vert v_1 - w_1 \Vert$. So by Proposition \ref{oinfty}, $v_1 \perp_{\infty} w_1$ and hence $v \perp_{\infty}^e w$. 
\end{proof}
\begin{remark}
	Let $(V, e)$ be an order unit space and assume that $u, v, w \in S_V$ with $u = v + w$ such that $u$ is an order projection in $V$. Expanding the proof of Proposition \ref{sum}, we can show that $v$ and $w$ are order projections in $V$ if and only if $v \perp_{\infty}^e w$.
\end{remark}

\section{An illustration}

Let $X$ be a real normed linear space. Consider $\widetilde{X} := \mathbb{R} \oplus_1 X$ and define 
$$\widetilde{X}^+ := \left\lbrace (\alpha, x) \in \widetilde{X}: x \in X ~ \mbox{and} ~ \Vert x \Vert \le \alpha \right\rbrace.$$ 
It is now a folklore that $(\widetilde{X}, \widetilde{X}^+, e)$ is an order unit space called an order unit space obtained by adjoining an order unit to a normed linear space where $e := (1, 0)$. (See, for example \cite{K21} and references therein.) 
\begin{lemma}
	$S_{\widetilde{X}} = \left\lbrace (\frac 12, x): \Vert x \Vert = \frac 12 \right\rbrace$.
\end{lemma}
\begin{proof}
	Let $(\alpha, x) \in S_{\widetilde{X}}$. Then $(\alpha, x) \in \widetilde{X}^+$ with $\Vert (\alpha, x) \Vert = 1 = \Vert e - (\alpha, x) \Vert$. Thus $\Vert x \Vert \le \alpha$, $\Vert x \Vert = 1 - \alpha$ and $\Vert - x \vert = 1 - (1 - \alpha)$. In other words, $\Vert x \Vert = \frac 12 = \alpha$. This completes the proof as the verification of the converse part is easy.
\end{proof}
\begin{proposition}\label{aoup}
	Let $X$ be a non-zero real normed linear space and let $x \in X$ with $\Vert x \Vert = \frac 12$. Then $(\frac 12, x)$ has the OUP in $\widetilde{X}$ if and only if $\tilde{X}_{(\frac 12, x)} = \mathbb{R} (\frac 12, x)$. 
\end{proposition}
\begin{proof}
	Let $(\frac 12, x)$ have the OUP in $\widetilde{X}$. First, we assume that $(\alpha, z) \in \widetilde{X}$ with $\vert \alpha \vert + \Vert z \Vert := \Vert (\alpha, z) \Vert = 1$ is such that $- (\frac 12, x) \le (\alpha, z) \le (\frac 12, z)$. Then $- \frac 12 \le \alpha \le \frac 12$, that is, $\vert \alpha \vert \le \frac 112$. Since $(\frac 12, x)$ has the OUP in $\widetilde{X}$, by Theorem \ref{oup}, we have $(e - (\frac 12, x)) \perp_{\infty} (\alpha, z)$. In other words, $(\frac 12, - x) \perp_{\infty} (\alpha, z)$. Thus by Proposition \ref{oinfty}, we get $\Vert (\frac 12, - x) \pm (\alpha, z) \Vert = 1$. Now it follows that 
	$$\left\vert \frac 12 + \alpha \right\vert + \Vert - x + z \Vert = 1 = \left\vert \frac 12 - \alpha \right\vert + \Vert - x - z \Vert.$$ 
	Since $\vert \alpha \vert \le \frac 12$, we obtain that $\Vert x - z \Vert = \frac 12 - \alpha$ and $\Vert x + z \Vert = \frac 12 + \alpha$. Thus 
	$$2 \Vert z \Vert = \Vert (x - z) - (x + z) \Vert \le \Vert x - z \Vert + \Vert x + z \Vert = 1.$$ 
	Hence $1 - \vert \alpha \vert = \Vert z \Vert \le \frac 12$, that is, $\vert \alpha \vert \ge \frac 12$. Since $\vert \alpha \vert \le \frac 12$, we get $\vert \alpha \vert = \frac 12$. 
	
	If $\alpha = \frac 12$, then $\Vert x - z \Vert = 0$ so that $z = x$ and consequently, $(\alpha, z) = (\frac 12, x)$. 
	
	If $\alpha = - \frac 12$, then $\Vert x + z \Vert = 0$ so that $z = - x$ and consequently, $(\alpha, z) = - (\frac 12, x)$. 
	
	Now, if $(\alpha, z) \in \widetilde{X}_{(\frac 12, x)}$ with $(\alpha, z) \ne (0, 0)$, then by the OUP of $(\frac 12, x)$ in $\widetilde{X}$, we get $\Vert (\alpha, z) \Vert^{-1} (\alpha, z) \in \left[ - (\frac 12, x), (\frac 12, x) \right]$. Thus as above $(\alpha, z) = \Vert (\alpha, z) \Vert (\frac 12, x)$ or $(\alpha, z) = - \Vert (\alpha, z) \Vert (\frac 12, x)$. In either case, $(\alpha, z) \in \mathbb{R} (\frac 12, x)$. Hence $\widetilde{X}_{(\frac 12, x)} = \mathbb{R} (\frac 12, x)$. Since $(\frac 12, - x) \perp_{\infty} (\frac 12, x)$, the converse is straight forward. 
\end{proof} 
\begin{remark}
	Note that $(\frac 12, x)$ is an order projection in $\widetilde{X}$ if and only if $(\frac 12, x)$ and $(\frac 12, - x)$ have the OUP in $\widetilde{X}$. Thus by Proposition \ref{aoup}, $(\frac 12, x)$ is an order projection in $\widetilde{X}$ if and only if $\tilde{X}_{(\frac 12, x)} = \mathbb{R} (\frac 12, x)$ and $\tilde{X}_{(\frac 12, - x)} = \mathbb{R} (\frac 12, - x)$.
\end{remark}
\begin{proposition}\label{strict}
	Let $X$ be a non-zero normed linear space. Then $(x, \frac 12)$ has the OUP in $\widetilde{X}$ for every $x \in X$ with $\Vert x \Vert = \frac 12$ if and only if $X$ is strictly convex.
\end{proposition} 
\begin{proof}
	First, we assume that $X$ is strictly convex. Let $x \in X$ with $\Vert x \Vert = \frac 12$ and assume that $- (\frac 12, x) \le (\alpha, z) \le (\frac 12, x)$ for some $z \in X$. Then $(\frac 12 + \alpha, x + z), (\frac 12 - \alpha, x - z) \in \widetilde{X}^+$ so that 
	$$\Vert x + z \Vert \le \frac 12 + \alpha ~ \mbox{and} ~ \Vert x - z \Vert \le \frac 12 - \alpha.$$ 
	Now 
	\begin{eqnarray*}
		1 = 2 \Vert x \Vert &=& \Vert (x + z) + (x - z) \Vert \\ 
		&\le& \Vert x + z \Vert + \Vert x - z \Vert \\ 
		&\le& \frac 12 + \alpha + \frac 12 - \alpha = 1.
	\end{eqnarray*} 
	Thus 
	$$\Vert x + z \Vert = \frac 12 + \alpha ~ \mbox{and} ~ \Vert x - z \Vert = \frac 12 - \alpha \qquad (*) $$ 
	and we have 
	$$\Vert x + z \Vert + \Vert x - z \Vert = \frac 12 + \alpha + \frac 12 - \alpha = 1 = \Vert 2 x \Vert. \qquad (**)$$ 
	We show that $z = 2 \alpha x$. 	If $z = x$, then $\alpha = \frac 12$ and if $z = - x$, then $\alpha = - \frac 12$. In both the cases, we have $z = 2 \alpha x$. Thus we assume that $z \ne x$ and $z \ne - x$. Since $(x + z) + (x - z) = 2 x$, by $(**)$ using the strict convexity in $X$, we get that 
	$$\frac{x + z}{\Vert x + z\Vert} = \frac{x - z}{\Vert x - z\Vert}.$$ 
	If we  simplify using $(*)$, we again get $z = 2 \alpha x$. Thus in all case we have $(\alpha, z) = 2 \alpha (\frac 12, x)$. Hence $\langle (\frac 12, x) \rangle = \mathbb{R} (\frac 12, x)$. Now, by Corollary \ref{aoup}, we may conclude that $(\frac 12, x)$ has the OUP in $\widetilde{X}$ whenever $x \in X$ with $\Vert x \Vert = \frac 12$. 
	
	Next, we assume that $X$ is not strictly convex. Then we can find $y, z \in X$ with $y \ne z$ and $\Vert y \Vert = \frac 12 = \Vert z \Vert$ such that $\Vert y + z \Vert = 1$. Put $x = \frac 12 (y + z)$. Then $\Vert x \Vert = \frac 12$ and $(0, 0) \le (\frac 14, \frac 12 z) \le (\frac 12, x)$ but $(\frac 14, \frac 12 z) \notin \mathbb{R} (\frac 12, x)$. Thus $(\frac 12, x)$ does not possess the OUP in $\widetilde{X}$. 
\end{proof}
\begin{remark}
	In statement of Proposition \ref{strict}, we can replace the phrase ``$(x, \frac 12)$ has the OUP'' by the phrase ``$(x, \frac 12)$ is an order unit''.
\end{remark}
%\thanks{{\bf Acknowledgements:} 

\end{document}